\documentclass[reqno]{article}
\usepackage{wrapfig}
\usepackage{pstricks}
\usepackage{multicol}
\usepackage{graphicx}
\usepackage{amsthm}
\usepackage{amsfonts}
\usepackage{amsmath}
\usepackage{mathrsfs}
\usepackage{amssymb}
\usepackage{hyperref}
\usepackage{mathscinet}
\usepackage[margin=1in]{geometry}
\usepackage{pictexwd,dcpic}
\theoremstyle{definition}
\newtheorem{theorem}{Theorem}
\newtheorem{lemma}[theorem]{Lemma}

\newtheorem{example}[theorem]{Example}
\newtheorem{proposition}[theorem]{Proposition}
\newtheorem{corollary}[theorem]{Corollary}
\newtheorem{definition}{Definition}

\def\N{\mathbb{N}}

\begin{document}

\title{Biologically unavoidable sequences}
\author{Samuel A.~Alexander\thanks{Email:
alexander@math.ohio-state.edu}\hphantom{*}\\
\emph{Department of Mathematics, the Ohio State University}}

\maketitle

\vspace{-2.525in}
\vspace{-10pt}
\begin{center}
\resizebox*{15cm}{1.731cm}{\input{unavheadtest.supplement}}
\end{center}
\vspace{-10pt}
\vspace{-.9in}
\vspace{2.5 in}

\begin{abstract}
A biologically unavoidable sequence
is an infinite gender sequence which occurs in every gendered, infinite 
genealogical network satisfying certain tame conditions.
We show that every eventually periodic sequence is biologically 
unavoidable (this generalizes K\"{o}nig's Lemma), and we exhibit some 
biologically avoidable sequences.
Finally we give an application of unavoidable sequences
to cellular automata.
\end{abstract}

\section{Introduction}
\label{introsection}

The following definition is motivated by biological considerations.
The idea of modelling the biosphere with a directed graph goes at least 
back to Hennig \cite{hennig}, more recently (and more formally)
to Dress, Moulton, Steel and 
Wu \cite{dress}.  The idea of including vertices for future organisms was 
made explicit in Kornet, Metz, and Schellinx 
\cite{kornet1995}.  The (simplifying) assumption of infinitely many vertices was 
made explicit in Alexander \cite{alexander}.
We hope to submit the results of this paper as an answer to Sturmfels \cite{sturmfels}.

\begin{definition}
\label{populations}
Let $n>0$.
An \emph{infinite $n$-gendered population} is a directed graph $G=(V,E)$, 
together with a map $v\mapsto t(v)\in\mathbb{R}$ assigning 
\emph{birthdates} to vertices and a map $e\mapsto g(e)\in\{1,\ldots,n\}$ 
assigning \emph{genders} to edges, satisfying the following conditions (we 
call $v$ a \emph{parent} of $u$ if $(v,u)\in E$, we define \emph{child} 
similarly, and we define \emph{ancestor} and \emph{descendant} in the 
obvious (strict) way):
\begin{itemize}
\item (A1) There are only finitely many parentless vertices (call them 
\emph{roots}).
\item (A2) Each vertex has only finitely many children.
\item (A3) For every $r\in\mathbb{R}$, $\{v\in V\,:\,t(v)\leq r\}$ is 
finite, and for each $(u,v)\in E$, $t(u)<t(v)$.
\item (A4) $|V|=\infty$.
\item ($n$-Gendered) For every non-root $u$ and
every $1\leq i\leq n$, $u$ has a parent $v$ with $g(v,u)=i$.
\end{itemize}
\end{definition}

Note that A3 implies $G$ is a DAG and is \emph{reverse-well-founded}, i.e.,
has no infinite reverse-directed path.
Thus every non-root is a descendant of some root.
Although Definition~\ref{populations} stipulates that edges be gendered, 
we will often consider special cases where vertices are gendered, 
implicitly giving each edge the gender of its initial vertex.

\begin{definition}
Let $n>0$.
An infinite sequence $s=(s_1,s_2,\ldots)\in\{1,\ldots,n\}^{\mathbb{N}}$
is \emph{biologically unavoidable} if every infinite $n$-gendered 
population \emph{realizes} $s$-- by which we mean there is a
vertex sequence $v_1,v_2,\ldots$ with each $v_i$ a 
parent of $v_{i+1}$ and each $g(v_i,v_{i+1})=s_i$.  If $s$ is not 
biologically unavoidable, it is \emph{biologically avoidable}.
\end{definition}

A priori, biological unavoidability appears ill-defined.  The following lemma shows that it is well-defined.

\begin{lemma}
Suppose $s\in\{1,\ldots,n\}^{\mathbb{N}}$ and at the same
time $s\in \{1,\ldots,m\}^{\mathbb{N}}$.
Then $s$ is realized in every infinite $n$-gendered population if and only if
$s$ is realized in every infinite $m$-gendered population.
\end{lemma}

\begin{proof}
Without loss of generality, $n<m$.
If there is an infinite $m$-gendered population which fails to realize $s$,
delete all edges with genders $>n$ to obtain an infinite $n$-gendered population
which fails to realize $s$.
Conversely, suppose there is an $n$-gendered population $P$ failing to realize $s$.
Inductively it suffices to show there's an $(n+1)$-gendered population failing to realize $s$.
Let $P'$ be a disjoint copy of $P$.  To the graph $P'\sqcup P$, add two $(n+1)$-gendered edges,
$(v',u)$ and $(v,u'$), for every $n$-gendered edge $(v,u)$ in $P$.
It's easy to see this makes $P\sqcup P'$ an $(n+1)$-gendered population
which fails to realize $s$.
\end{proof}

In Section \ref{eventuallyperiodicsection}, we will show that every eventually periodic sequence is 
biologically unavoidable.

In Section \ref{existencesection}, we will establish the existence of biologically avoidable 
sequences.  This nontrivial fact involves a couple of unexpected enumerative
combinatorial arguments.

In Section \ref{juliansection}, we will show that there are biologically avoidable 
sequences from $\{1,2\}^{\mathbb{N}}$ which contain no gender repeated 
thrice in a row.

In Section \ref{johnstonsection}, we will give an unexpected application to cellular 
automata: an alternate proof of a result about spaceship speed limits in 
Conway's Life-like Games, first proved by Nathaniel Johnston 
\cite{johnston}.

\section{Eventually periodic implies biologically unavoidable}
\label{eventuallyperiodicsection}

Before proving the unavoidability of eventually periodic sequences, a 
small amount of machinery must be developed.  For the remainder of this 
section, let $P=(V,E)$ be an infinite $n$-gendered population, $n>0$.

\begin{definition}
For $i\in\mathbb{N}$,
we define the set $V_i\subseteq V$ as follows:
a vertex $u\in V$ lies in $V_i$ if and only if there is
some root $r$ and some directed path from $r$ to $u$ of length $\leq i$.
\end{definition}

Thus $V_0$ is the set of roots, $V_1$ contains the roots and their children, 
and so on.

\begin{lemma}
\label{Vnfinite}
For every 
$i\in\mathbb{N}$, $V_i$ is finite.
\end{lemma}

\begin{proof}
Follows trivially from A1 and A2.
\end{proof}



\begin{multicols}{2}

\begin{lemma}
For every $1\leq i\leq n$, there is a map
$i^*:V\backslash V_0\to V$
such that for every
$u\in V\backslash V_0$ we have $(i^*u,u)\in E$ and $g(i^*u,u)=i$.
\end{lemma}

\begin{proof}
By the axiom of choice and the $n$-Gendered assumption on $P$.
\end{proof}

In case $n=2$, one might refer to $1^*$ and $2^*$ as \emph{motherhood} and \emph{fatherhood} maps, in some order.
We will write $i^*j^*$ for $i^*\circ j^*$.

\columnbreak

\[
\begindc{\commdiag}[5]
\obj(-10,10)[aa]{$V\backslash V_0$}
\obj(10,10)[bb]{$E$}
\obj(-10,0)[cc]{$\{i\}$}
\obj(10,0)[dd]{$\{1,\ldots,n\}$}
\mor{aa}{bb}{$(i^*,\mbox{Id})$}
\mor{aa}{cc}{}
\mor{bb}{dd}{$g$}
\mor{cc}{dd}{}[\atright,\injectionarrow]
\enddc\]
\end{multicols}

\begin{definition}
\label{PiSubS}
If $s=(s_1,s_2,\ldots)$ is a periodic sequence on $\{1,\ldots,n\}$, with 
period $p$, and if $u\in V$,
then we define the \emph{$s$-path to $u$}, a finite directed path, written 
$s^*u$, as follows.
\begin{enumerate}
\item If $u\in V_{p-1}$ then $s^*u=(u)$ (the length $0$ path).
\item Otherwise, \[s^*u
= ( s^*_1\cdots s^*_{pq}u, \,\,\, s^*_2\cdots s^*_{pq}u, \,\,\, \ldots, \,\,\, s^*_{pq -1}s^*_{pq}u, \,\,\, s^*_{pq}u, \,\,\, u ),\]
where $q$ is maximal such that 
$s^*_1\cdots s^*_{pq}u$ is defined.
(Some such $q>0$ exists lest $u$ lie in $V_{p-1}$, and such $q$ are 
bounded above due to Assumption A3.)
\end{enumerate}
\end{definition}

For example, suppose $n=2$ and $\{1,2\}=\{M,F\}$.  If $s=(M,M,M,\ldots)$, then $s^*u$ is obtained as follows:
start with $u$, go to $u$'s father, go to \emph{his} father, and so on until a root is reached; then reverse the order
of the resulting path.  If $s=(M,F,F,M,F,F,\ldots)$, then $s^*u$ is obtained by starting at $u$, taking his mother,
taking her mother, taking her father, and repeating this three-step process until too close to a root to continue;
and then reversing the resulting path.

\begin{lemma}
\label{PiSubSWorks}
Let $s$, $p$ be as in Definition~\ref{PiSubS}.  For any $u\in V$, 
$s^*u$ is a finite directed path
starting at a vertex in $V_{p-1}$ and ending at $u$, and if $s^*u=(v_1,\ldots,v_r)$
then for all $1\leq i<r$, $g(v_i,v_{i+1})=s_i$.
\end{lemma}

\begin{proof}
To see that $s^*u$ starts in $V_{p-1}$, write $s^*u=(v_1,\ldots,v_r)$
and assume $v_1\not\in V_{p-1}$.
Then $s^*_1\cdots s^*_p v_1$ is defined.
Let $q$ be as in the definition of $s^*u$.
By periodicity of $s$,
\begin{align*}
s^*_1\cdots s^*_p v_1 &= s^*_1\cdots s^*_p (s^*_1 \cdots s^*_{pq} u)\\
&=
s^*_1\cdots s^*_p s^*_1 \cdots s^*_{pq} u\\
&=
s^*_1\cdots s^*_{p(q+1)} u,
\end{align*}
violating the maximality of $q$.
The rest of the lemma is clear.
\end{proof}

\begin{proposition}
\label{periodicunav}
Every periodic sequence is biologically unavoidable.
In fact, if $s$ is a sequence with period $p$,
it is realized by a path in $P$ which begins at a vertex in $V_{p-1}$.
\end{proposition}

\begin{proof}
Consider the set (call it $X$) of finite directed paths
$\{s^*u\}_{u\in V}$ (note $|X|=\infty$ by Assumption A4).
Each begins (by Lemma\ \ref{PiSubSWorks}) in $V_{p-1}$,
and $V_{p-1}$ is finite by Lemma~\ref{Vnfinite}.
Thus, there is some $u_1\in V_{p-1}$
such that 
infinitely many members (call them $X_1$) of $X$ begin at $u_1$.

Inductively, suppose
I've defined a finite directed path $u_1,\ldots,u_k$
and an infinite set $X_k\subseteq X$ such that
\begin{enumerate}
\item For all $\Pi\in X_k$, the $k$th vertex in $\Pi$ is $u_k$; and
\item $g(u_i,u_{i+1})=s_i$ for all $0<i<k$.
\end{enumerate}

It follows from Assumptions A1 and A2 that $X_k$
contains only finitely many paths of length $\leq k$,
we may assume it contains no paths so short.

Since each path in $X_k$ has $u_k$ as $k$th vertex,
each path in $X_k$ has some child of $u_k$ as$(k+1)$th 
vertex.  By Assumption A2, $u_k$ only has finitely many children.  Thus 
there 
is a child $u_{k+1}$ of $u_k$ such that an infinite subset 
$X_{k+1}$ of paths in $X_k$ have $u_{k+1}$ as $(k+1)$th 
vertex.

In particular, $X_{k+1}$ has at least one path, $s^*v$ 
for some $v\in V$.
By Lemma~\ref{PiSubSWorks},
$g(u_k,u_{k+1})=s_{k}$.

This inductively defines $u_1,u_2,\ldots$ with all the desired properties.
\end{proof}

It is also possible to prove Proposition~\ref{periodicunav}
using the \emph{compactness theorem} from first-order logic.
Compare the proof (see Simpson's book \cite{simpson}) that
weak K\"{o}nig's Lemma is equivalent (over $\mbox{RCA}_0$)
to the compactness theorem.

\begin{theorem}
\label{eventualperiodicimpliesunavoidable}
Every eventually periodic sequence is biologically unavoidable.
\end{theorem}

\begin{proof}
If $s$ is eventually periodic, write it as $s=t\frown t'$
where $t$ is finite of length $k$ and $t'$ is periodic.
Let $P'$ be the population obtained from $P$ by discarding all vertices in 
$V_{k-1}$; it's easy to see $P'$ remains an infinite gendered population.
By Proposition~\ref{periodicunav}, there is a directed path in $P'$ 
realizing 
$t'$.
This defines a path $u_1,u_2,\ldots$ in $P$, realizing $t'$, 
and avoiding $V_{k-1}$.
Back-extend this path to
\[
(t^*_1\cdots t^*_k u_1, \,\,\, t^*_2\cdots t^*_k u_1, \,\,\, \ldots, \,\,\, t^*_k u_1, \,\,\, u_1, \,\,\, u_2, \,\,\, \ldots),\]
which realizes $s$ as desired; this is possible because if not, that would 
imply $u_1\in V_{k-1}$.
\end{proof}

This generalizes K\"{o}nig's Lemma for trees, which can be seen as the
$1$-gender case of Theorem~\ref{eventualperiodicimpliesunavoidable}
with the additional constraint that vertices have only one parent.
The following corollary pushes this idea even further.

\begin{corollary}
There is a subset $V_0\subseteq V$, ancestrally closed (whenever $v\in V_0$ and $u$ is an ancestor of $v$ then $u\in V_0$),
such that for every eventually periodic sequence $s\in\{1,\ldots,n\}^{\mathbb{N}}$,
$G$ realizes $s$ with a path $p$ entirely in $V_0$, with the additional property that every vertex in $V_0$ has a descendant on $p$.
\end{corollary}

\begin{proof}
By Theorem~\ref{eventualperiodicimpliesunavoidable} above, combined with Theorem 3 and Proposition 5 from Alexander \cite{alexander}.
\end{proof}

Using Theorem~\ref{eventualperiodicimpliesunavoidable} we can give a game-theoretical characterization
of unavoidable sequences using the notion of \emph{guessability} discovered by Wadge \cite{wadge} (pp.~141--145)
(and independently by Alexander \cite{alexanderguessing}).
Let $s$ be a sequence on $\{1,\ldots,n\}$.  In the game $G_s$, $I$ starts by playing an infinite $n$-gendered
population $P$.  Thereafter, $II$ plays a path $p$ in $P$ (one vertex per turn) and $I$ tries to guess
(making one guess per turn)
whether $p$'s genders have the form $t\frown s$ for some finite $t$.
$I$ wins if $I$'s guesses converge to the correct answer, $II$ wins otherwise.
We leave it an exercise that $I$ has a winning strategy iff $s$ is biologically avoidable.
(This holds whether or not $II$ can see $I$'s guesses.)

\section{Existence of Biologically Avoidable Sequences}
\label{existencesection}

One might hope to cleverly generalize the argument from the previous 
section to non-periodic gender sequences.  In this section we'll show 
that's impossible.  There are sequences which are biologically avoidable.
Populations lacking certain gender sequences are analogous to Aronszajn 
trees (first introduced in \cite{kurepa}) in the sense that both provide
counterexamples to plausible-seeming generalizations of K\"{o}nig's Lemma.

In this and the next section we restrict attention to populations
with gendered vertices, implicitly gendering edges according to their 
initial vertices.

\begin{definition}
Suppose $P=(V,E)$ is an infinite $n$-gendered population, 
$s=(s_1,\ldots,s_k)$ is a
finite 
sequence,
$k>0$,
and $V$ is partitioned into \emph{heights} $H_1,H_2,\ldots$.
We say that 
$s$ is 
\emph{impossible in $P$ at height $k$} if there is no finite 
directed path $v_1,\ldots,v_k$, gendered by $s$,
with $v_1\in H_k$.
\end{definition}

\begin{lemma}
\label{abstractnonsense}
Suppose $s$ is an infinite $\{1,\ldots,n\}$-sequence.
If there is an infinite $n$-gendered population $P=(V,E)$ and a partition 
of $V$ 
into heights $H_1,H_2,\ldots$ such that for every $k>0$,
some finite restriction of $s$ is impossible in
$P$ at 
height $k$, then $s$ is biologically avoidable.
\end{lemma}

\begin{proof}
If $s$ were realized by $P$, it would be realized by some path, starting 
with a vertex in some height $H_k$, yet $s$ would have some finite 
restriction 
impossible at height $k$, impossible.
So $s$ is not realized by $P$, so $s$ is biologically avoidable.
\end{proof}


We will now define a specific family of infinite $2$-gendered populations 
(generalizing an example suggested by Timothy J.\ Carlson)
designed to take advantage of Lemma~\ref{abstractnonsense}.
Let $\{M,F\}=\{1,2\}$, we will refer to $M$-gendered vertices as 
\emph{males}, $F$-gendered vertices as \emph{females}, and adopt 
terminology such as \emph{son}, \emph{daughter} with the obvious meanings.

\begin{definition}
\label{carlsontree}
Suppose $h:\N^+\to\N^+$.
The infinite $2$-gendered population $T_h$ is defined as follows (Figure 1 
shows $T_{n\mapsto n}$).
The vertices of $T_h$ are partitioned into successive \emph{generations} 
$G_1,G_2,\ldots$, the $n$th generation consisting of
$h(n)$ males $m^n_1,\ldots,m^n_{h(n)}$
and $h(n)$ females $f^n_1,\ldots,f^n_{h(n)}$.
These vertices are given birthdates, $v\mapsto t(v)$, so that vertices in 
$G_i$ are born 
before those in $G_j$ whenever $i<j$, and within each generation $G_n$, 
$\max\{t(m^n_i),t(f^n_i)\}<\min\{t(m^n_j),t(f^n_j)\}$
whenever 
$i<j$.
Edges are defined as follows.
\end{definition}

\setlength{\columnsep}{.5in}
\begin{multicols}{2}

\begin{itemize}
\item $m^1_1$ and $f^1_1$ have no parents.
\item $\forall n>0$, $m^{n+1}_1$ has parents $m^n_{h(n)}$ and $f^n_1$.
\item $\forall n>0$, $f^{n+1}_1$ has parents $f^n_{h(n)}$ and $m^n_1$.
\item $\forall n>0$, $\forall 0<i<h(n)$,
$m^n_{i+1}$ has parents $m^n_i$ and $f^n_1$.
\item $\forall n>0$, $\forall 0<i<h(n)$,
$f^n_{i+1}$ has parents $f^n_i$ and $m^n_1$.
\end{itemize}

\columnbreak

\vphantom{A}

\begin{center}
\resizebox*{7cm}{1.701cm}{\input{carlsonsmalltest.supplement}}
\end{center}
\vspace{-.75cm}
\begin{quote}
Figure 1:
The infinite $2$-gendered population $T_{n\mapsto n}$.
Solid and open vertices correspond to males and females,
not necessarily in that order.
\end{quote}

\end{multicols}

\begin{lemma}
\label{lefttoreader}
Let $h:\N^+\to\N^+$.
\begin{enumerate}
\item $T_h$ is an infinite $2$-gendered population with roots $m^1_1$ and 
$f^1_1$.
\item If $m^n_i$ has a daughter, or $f^n_i$ has a son, then $i=1$.
\item No edge in $T_h$ skips an entire generation: if an edge has initial 
vertex in $G_i$, then it has terminal vertex in either $G_i$ or $G_{i+1}$.
\item If $B=(v_1,\ldots,v_n)$ is a finite directed path of males from 
$T_h$, such that $v_1\in G_i$ and $v_n\in G_j$, then for every $i<k<j$, 
$B$ contains all the males in $G_k$.
\end{enumerate}
\end{lemma}

\begin{proof}
Left to the reader.
\end{proof}

\begin{definition}
When a function $h:\N^+\to\N^+$ is clear from 
context, we let $\widehat{m}_1,\widehat{m}_2,\ldots$
denote the males in $T_h$ (over all the generations), ordered ascending by 
birthdate.  Similarly for $\widehat{f}_1,\widehat{f}_2,\ldots$.
We partition $T_h$ into \emph{heights} $H_1,H_2,\ldots$
by letting each $H_i=\{\widehat{m}_i,\widehat{f}_i\}$.
\end{definition}

The following technical lemma should be compared
and contrasted with Cauchy's polygonal 
number theorem (see \cite{nathanson}) which states that every positive 
integer can be written as a sum of $n$ $n$-gonal numbers, for any $n\geq 
3$.  For example, every natural number has the form 
$\sum_{p=1}^{b_1}p+\sum_{p=1}^{b_2}p+\sum_{p=1}^{b_3}p$.
Also worth comparing is the work \cite{cantor} of D.~Cantor and 
B.~Gordin, and more recently \cite{gupta} of S.~Gupta.

\begin{lemma}
\label{techlemma}
Suppose $h:\N^+\to\N^+$ and $\lim_{n\to\infty} h(n)=\infty$.
For every $u\in\N$, there is some positive integer $e$ which is \emph{not} 
of the form
\[
a+\sum_{p=1}^b h(c+p)
\]
for any $a,b,c\in\N$ with $c\leq u$ and $a\leq\max\{h(1),\ldots,h(u)\}$.
\end{lemma}

\begin{proof}
Let $A=\max\{h(1),\ldots,h(u)\}$.
Since $\lim_{n\to\infty}h(n)=\infty$, there is some $M_0$ so big that 
$h(c+M)>(u+1)(A+1)$ whenever $M\geq M_0$ and $0\leq c\leq u$.
Let
\[
M=\max\left\{ a+\sum_{p=1}^b h(c+p)\,:\, a\leq A, c\leq u, b\leq M_0
\right\}
\]
and let $X=\{M+1,M+2,\ldots,M+(u+1)(A+1)+1\}$.

I claim that for every $c\leq u$, $X$ contains at most $A+1$ different 
numbers of the form $a+\sum_{p=1}^b h(c+p)$ with $a\leq A$.
If not, by the pigeonhole principle there is some particular $a\leq A$
and some $b_1<b_2$ such that $a+\sum_{p=1}^{b_1}h(c+p)$ and
$a+\sum_{p=1}^{b_2}h(c+p)$ are both in $X$, let $d$ be their difference.
Then $d=h(c+b_1+1)+\cdots+h(c+b_2)\geq h(c+b_2)$.
Since $a+\sum_{p=1}^{b_2}h(c+p)>M$ (by virtue of being in $X$), by 
definition of $M$ this implies 
$b_2\geq M_0$, whereby $d\geq h(c+b_2)>(u+1)(A+1)$.  This is absurd: $X$ 
is 
made up of $(u+1)(A+1)+1$ consecutive points, no two of them can have a 
difference $>(u+1)(A+1)$.  The claim is proved.

Given the above claim, the number of numbers in $X$ with the form 
$a+\sum_{p=1}^b h(c+p)$, $a\leq A$, $c\leq u$, is at most $(u+1)(A+1)$: 
$u+1$ choices for $c$, times $A+1$ numbers of the given form for each $c$.
Since $|X|>(u+1)(A+1)$, $X$ contains an $e$ as desired.
\end{proof}

\begin{proposition}
\label{stepbystep}
If $h:\N^+\to\N^+$, $\lim_{n\to\infty} h(n)=\infty$,
$k>0$, and $s$ is any finite $\{M,F\}$-sequence,
there is some $e>0$
such that $s\frown M^eF$ is impossible in $T_h$ at height $k$.
\end{proposition}

\begin{proof}
We may assume $s$ nonempty.
We may also assume the first gender in $s$ is $M$, the other case being 
similar.
Let $\ell=\mbox{length}(s)$.
By Lemma~\ref{techlemma},
there is some $e>1$ such that $e-1$ is not of the form
$a+\sum_{p=1}^b h(c+p)$ for any $a,b,c\in\N$ with $c\leq \ell+k$
and $a\leq \max\{h(1),\ldots,h(\ell+k)\}$.
Thus $e$ itself is not of the form $a+1+\sum_{p=1}^b h(c+p)$
for any such $a,b,c$.
We will show $s\frown M^eF$ is impossible in $T_h$ at height $k$.
Suppose not: suppose there is a finite directed path 
$v_0,\ldots,v_{\ell},v_{\ell+1},\ldots,v_{\ell+e}$ in $T_h$, gendered by 
$s\frown M^eF$, with $v_0=\widehat{m}_k$.

We would like to estimate in which
generation does $v_{\ell}$, the first vertex 
corresponding to the $M^eF$ block, lie.  We will be content with an 
overestimate.
By Lemma~\ref{lefttoreader} part 3, every edge either ends in the same 
generation where it began, or at most one generation further.
Thus $v_0=\widehat{m}_k$ is in at most the $k$th generation, and 
$v_{\ell}$ is in at most the $(k+\ell)$th generation.

I claim $e$ has the form $a+1+\sum_{p=1}^b h(c+p)$ for some $a,b,c\in\N$ 
with $c\leq \ell+k$, $a\leq \max\{h(1),\ldots,h(\ell+k)\}$,
a contradiction.
Let $i,j$ be such that $v_{\ell}$ is in $G_i$ (so $i\leq k+\ell$)
and $v_{\ell+e-1}$ is in $G_j$.
Since $v_{\ell+e}$ is female, and $v_{\ell+e-1}$ is male, by 
Lemma~\ref{lefttoreader} part 2, $v_{\ell+e-1}$ must be $m^j_1$.
Thus $B=(v_{\ell},\ldots,v_{\ell+e-1})$ is a finite directed
path of male vertices beginning with $m^i_x$ for some $x$ and ending with 
$m^j_1$.
All of these males lie in generations between $i$ and $j$ inclusive,
and for every $i<p<j$, Lemma~\ref{lefttoreader} part 4 says that $B$ 
contains all the males in $G_p$.  Let us count the vertices in $B$:
\begin{enumerate}
\item The number of males from $G_i$ included in $B$ is at most all of 
them, so $B$ has $\leq h(i)$ vertices from $G_i$.
\item For $i<p<j$, $B$ contains all $h(p)$ males of $G_p$.
This is a total of
\[
h(i+1)+\cdots+h(i+(j-i-1))=\sum_{p=1}^{j-i-1}h(i+p)\]
vertices.
\item $B$ has exactly one vertex from $G_j$, namely $m^j_1$.
\end{enumerate}
Thus, the number $e$ of vertices in $B$ has the form $a+1+\sum_{p=1}^b 
h(c+p)$ for some natural $c\leq i\leq \ell+k$, some $a\leq 
\max\{h(1),\ldots,h(\ell+k)\}$,
and some natural $b$, the desired contradiction.
\end{proof}

\begin{corollary}
\label{thereisanunavoidable}
There is a biologically avoidable sequence.
\end{corollary}

\begin{proof}
Let $h$ be as in Definition~\ref{carlsontree}.
Using Proposition~\ref{stepbystep} repeatedly,
define finite sequences $\{s^k\}_{k>0}$
such that each $s^{k+1}$ strictly extends $s^k$
and each $s^k$ is impossible in $T_h$ at height $k$.
Then $s=\cup s^k$ is biologically avoidable
by Lemma~\ref{abstractnonsense}.
\end{proof}

\begin{example}
$M^9FM^{4200}F\cdots M^{e_n}F\cdots$ is biologically avoidable,
where each $e_n>0$ is chosen minimal so as to avoid the form 
$a+1+\sum_{p=1}^b(c+p)$ ($a,c\leq u$ where $u=\ell+n$, where $\ell$ is the 
length of $M^9F\cdots M^{e_{n-1}}F$).
\end{example}

The above example is suboptimal, because our proof of 
Proposition~\ref{stepbystep} used such staggering overestimates.

\begin{example}
An alternate way to obtain avoidable sequences is to follow the proof of 
Corollary~\ref{thereisanunavoidable}, but to obtain $s_{k+1}$, rather than 
follow the instructions in Proposition~\ref{stepbystep}, we can simply do 
a brute-force search to find the minimal $e>0$ such that $s_k\frown M^eF$ 
is impossible at height $k+1$
(Proposition~\ref{stepbystep} says we won't get stuck).  If we do this for 
the function $h(n)=n$, 
we obtain the avoidable sequence $M^3FM^5FM^8FM^{11}F\cdots$
where each block of $M$'s is $3$ longer than the last, with one exception 
at the beginning (the proof is tedious so we omit it).
That one exception is annoying, so here's how to further optimize the 
sequence:  choose each $s_k$ to be impossible at height $k+1$.
This leaves open the possibility, a priori, the sequence could occur in 
$T_h$
at height $1$.  However, by \emph{fortune}, it ends up being impossible at 
height $1$ anyway.  This yields a very nice well-behaved avoidable 
sequence,
\[
M^2FM^5F\cdots M^{3n-1}F\cdots
\]
(again we omit the formal proof).
\end{example}

All the avoidable sequences we obtain in this manner have the property
that they contain arbitrarily long blocks of one gender.

\section{There is a biologically avoidable sequence in which no block of 
males or females has length more than 2}
\label{juliansection}

Julian Ziegler Hunts discovered an interesting family of populations which 
witness the avoidability of sequences with very short blocks of males and 
females.

\begin{definition}
\label{huntstree}
Suppose $h:\N^+\to\N^+$.
The infinite $2$-gendered population $H_h$ is defined in the same way
as $T_h$ was defined (Definition~\ref{carlsontree})
except for its edges, which are instead defined as follows.
\end{definition}

\setlength{\columnsep}{.5in}
\begin{multicols}{2}

\begin{itemize}
\item $m^1_1$ and $f^1_1$ have no parents.
\item $\forall n>0$, $m^{n+1}_1$ has parents $m^n_1$ and $f^n_{h(n)}$.
\item $\forall n>0$, $f^{n+1}_1$ has parents $m^n_{h(n)}$ and
$f^n_{h(n)}$.
\item $\forall n>0$, $\forall 0<i<h(n)$, $m^n_{i+1}$ has parents
$m^n_1$ and $f^n_i$.
\item $\forall n>0$, $\forall 0<i<h(n)$,
$f^n_{i+1}$ has parents $f^n_i$ and $m^n_i$.
\end{itemize}

\columnbreak

\vphantom{A}

\begin{center}
\includegraphics[scale=.5]{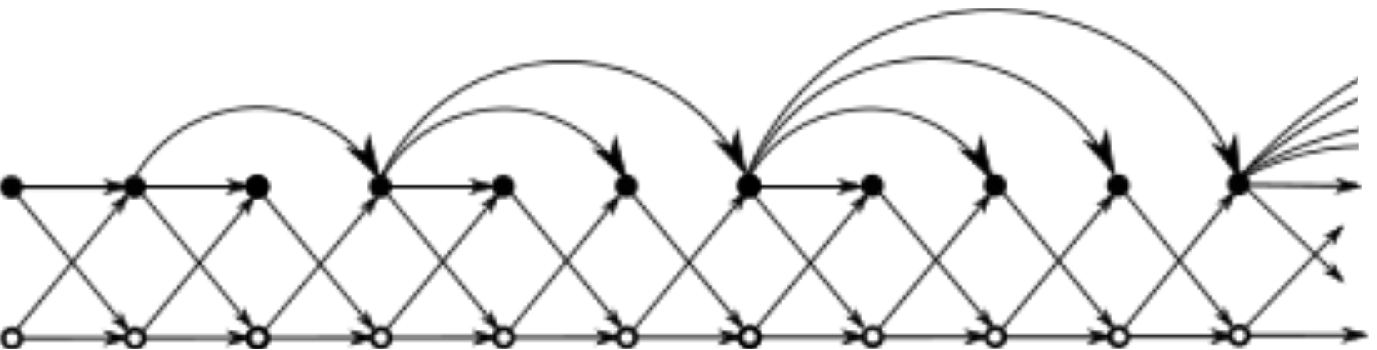}
\end{center}
\vspace{-.75cm}
\begin{quote}
Figure 2:
The infinite $2$-gendered population $H_{n\mapsto n}$.
\end{quote}

\end{multicols}

\begin{lemma}
\label{lefttoreaderhunts}
Let $h:\N^+\to\N^+$.
\begin{enumerate}
\item $H_h$ is an infinite $2$-gendered population with roots $m^1_1$ and 
$f^1_1$.
\item If $m^n_i$ has a son, then $i=1$.
\item No edge skips an entire generation: if an edge of $H_h$ has initial 
vertex in 
$G_i$ then its terminal vertex is in $G_i$ or $G_{i+1}$.
\item If $B=(v_1,\ldots,v_n)$ is a finite directed-path in $H_h$, whose 
genders are alternating, and if $v_1$ lies in $G_i$ and $V_n$ lies in 
$G_j$, then for every $i<p<j$ and every $0<k\leq h(p)$,
precisely one of $\{m^p_k,f^p_k\}$ appears in $B$.
\end{enumerate}
\end{lemma}

\begin{proof}
Left to the reader.
\end{proof}

We define the \emph{heights} of $H_h$, and the corresponding 
$\widehat{m}_i$ and $\widehat{f}_i$, in the same way as we did for $T_h$.

\begin{proposition}
Let $h:\N^+\to\N^+$ be such that $\lim_{n\to\infty}h(n)=\infty$
and $h(n)$ is even for every $n$.
If $k>0$ and
$s$ is any finite $\{M,F\}$-sequence,
then there is some $e>0$
such that $s\frown (FM)^eM$ is impossible in $H_h$ at height $k$.
\end{proposition}

\begin{proof}
We may assume $s$ nonempty and that the first gender in $s$ is $M$.
Let $\ell=\mbox{length}(s)$.  Since $h(n)$ is even for every $n$, by 
Lemma~\ref{techlemma}
there is an $e>1$ such that $e-1$ is not of the form $a+\sum_{p=1}^b 
\frac{1}{2}h(c+p)$ for any $a,b,c\in\N$
with $c\leq \ell+k$ and $a\leq \max\{h(1)/2,\ldots,h(\ell+k)/2\}$.
Thus $e$ itself is not of the form $a+1+\sum_{p=1}^b \frac{1}{2}h(c+p)$ 
for any such 
$a,b,c$.  We will show $s\frown (FM)^eM$ is impossible in $H_h$ at height 
$k$.
If not, there is a finite directed path
$v_0,\ldots,v_{\ell},v_{\ell+1},\ldots,v_{\ell+2e}$ in $H_h$, gendered by 
$s\frown (FM)^eM$, with $v_0=\widehat{m}_k$.

By similar reasoning to the proof of Proposition~\ref{stepbystep},
$v_{\ell}$ is in at most the $(k+\ell)$th generation.
I claim $e$ has the form
$a+1+\sum_{p=1}^b \frac{1}{2}h(c+p)$,
some 
$a,b,c$ as above,
a contradiction.
Let $i,j$ be such that $v_{\ell}\in G_i$ (so $i\leq k+\ell$)
and $v_{\ell+2e-1}\in G_j$.
Since $v_{\ell+2e-1}$ is male and has a son $v_{\ell+2e}$,
Lemma~\ref{lefttoreaderhunts} part 2 ensures $v_{\ell+2e-1}=m^j_1$.
Thus $B=(v_{\ell},\ldots,v_{\ell+2e-1})$ is a finite directed path of 
alternating gender beginning with $f^i_x$ for some $x$ and ending with 
$m^j_1$.  All these vertices lie in generations between $i$ and $j$ 
inclusive, and for each $i<p<j$, Lemma~\ref{lefttoreaderhunts} part 4
implies that the number of vertices in $B\cap G_p$ is exactly $h(p)$.
Count the \emph{male} vertices in $B$:
\begin{enumerate}
\item The number of males from $G_i$ is at most half of them (since $B$ 
alternates genders),
that is at most $h(i)/2$.
\item For any $i<p<j$, the number of males from $G_p$ is exactly half 
of them, $h(p)/2$, by Lemma~\ref{lefttoreaderhunts} part 4 since $B$ 
alternates genders.
\item There is exactly one male from $G_j$, namely 
$m^j_1$.
\end{enumerate}
Thus the number of males in $B$ is of the form
$a+1+\sum_{p=1}^b \frac{1}{2}h(c+p)$ ($a,b,c$ as above).
But the number of males in $B$ is $e$-- absurd.
\end{proof}

\begin{corollary}
There is a biologically avoidable sequence in which no
gender occurs thrice in a row.
\end{corollary}

\begin{proof}
Similar to the proof of Corollary~\ref{thereisanunavoidable}.
\end{proof}

\section{Application to cellular automata}
\label{johnstonsection}

In his paper \cite{johnston}, Nathaniel Johnston proved that in certain (a 
large family) of Conway's Life-like games, spaceships have an orthogonal 
speed limit of $1/2$ cells per generation, and a diagonal speed limit of 
$1/3$ cells per generation.  We will give an alternate proof 
using a 
technique which, we believe, might be generalizable to obtain results of a 
wide variety\footnote{If nothing else, our proof would generalize with 
minimal changes to Life-like games in higher dimensions (see Bays 
\cite{bays}).}.  We assume 
a novice-level familiarity with Life-like games (see \cite{eppstein}), 
and brush formal details under the rug.

\begin{definition}
Suppose a Life-like game is played, with some initial configuration in 
generation $1$, which generates a configuration in generation 2, and so 
on.  A \emph{lifeline} for this gameplay is a sequence $c_1,c_2,\ldots$ of 
cells such that:
\begin{enumerate}
\item Each $c_i$ is alive in generation $i$.
\item For each $i$, either $c_i=c_{i+1}$ or $c_i$ is adjacent to 
$c_{i+1}$.
\end{enumerate}
\end{definition}

Thus, a lifeline is a (not necessarily simple) stroll through the cells 
which, at each $i$th step, visits a cell alive in the $i$th generation.

\begin{lemma}
\label{forbidden}
(Two Forbidden Directions) Suppose $x,y$ are any two of the following 
\emph{directions}:
\[
\mbox{N, S, E, W, NE, NW, SE, and SW.}
\]
Consider a Life-like gameplay with the following properties:
\begin{enumerate}
\item The initial configuration is finite.
\item Each generation contains at least one live cell.
\item Birth requires $\geq 3$ neighbors and survival requires $\geq 1$ 
neighbor.
\end{enumerate}
For any such gameplay, there is a lifeline which never steps in direction 
$x$ or $y$ (that is, $x_{i+1}$ is never located in the $x$ or $y$ 
direction from $x_i$).
\end{lemma}

To be clear, the third condition is to be understood as liberally as 
possible, making the lemma apply not only to the ruleset 
$B345678/S12345678$, but to any 
sub-ruleset thereof.

We postpone the proof (using biologically unavoidable sequences) of 
Lemma~\ref{forbidden} so we can first see how it applies to spaceship 
speed limits.

\begin{theorem}
(Johnston 2010) For any Life-like ruleset where birth requires $\geq 3$ 
neighbors and survival requires $\geq 1$ neighbor, spaceships can move at 
most $\frac{1}{2}$ cells per generation orthogonally and at most 
$\frac{1}{3}$ cells per generation diagonally.
\end{theorem}

\begin{proof}[Proof sketch]
Since a spaceship is initially finite and does not go extinct, the 
hypotheses of Lemma~\ref{forbidden} are met.

(Orthogonal)  By symmetry, it's enough to show the spaceship cannot exceed
$\frac{1}{2}$ cells per generation northward.
By Lemma~\ref{forbidden}, there is a lifeline which never steps
in direction N or NE.  The cells in this lifeline are living cells,
hence cells in the spaceship, and it follows that the spaceship cannot 
travel faster than the lifeline.  The only movement the lifeline can make 
with northward component is NW (N and NE being forbidden).
Any such step also moves the lifeline westward, and so to maintain an 
overall northward direction, any such step must be compensated for by a 
step in one of the directions E or SE (NE is forbidden).
So at least two total steps are required per unit of overall northward 
movement.  Thus the speed limit, $\frac{1}{2}$.

(Diagonal)
By symmetry, it's enough to show the spaceship cannot exceed $\frac{1}{3}$ 
cells per generation northeastward.  By Lemma~\ref{forbidden}, there is a 
lifeline which never steps N or NE.  The only way the lifeline may move 
northward is by moving NW, and \emph{two} eastward steps must be added to 
produce overall NE movement.  Thus the speed limit, $\frac{1}{3}$.
\end{proof}

\begin{proof}[Proof sketch of Lemma~\ref{forbidden}]
Let $V$ be the set of pairs $(c,i)$ such that $c$ is a cell alive in 
generation $i$.  Direct an edge from $(c,i)$ to $(d,j)$ if $j=i+1$ and 
either $c=d$ or $c$ is a neighbor of $d$.
By (3) of Lemma~\ref{forbidden}, if $d$ is a cell born into generation 
$i+1$, there must have been at least three distinct neighbors 
$c_1,c_2,c_3$ of $d$ alive in generation $i$.  There are only two 
forbidden directions, so (possibly relabelling) we may assume $d$ does not 
lie in a forbidden direction from $c_1$.  Gender the edge
$((c_1,i),(d,i+1))$ male, and gender all other edges terminating in 
$(d,i+1)$ female.
If $d$ survives into generation $i+1$, then gender the edge 
$((d,i),(d,i+1))$ male and gender all other edges terminating in $(d,i+1)$ 
female.  Let each vertex $(c,i)$ have birthdate $i$.
The reader may check (using the hypotheses of Lemma~\ref{forbidden})
this makes $V$ an infinite $2$-gendered population.
By Theorem~\ref{eventualperiodicimpliesunavoidable},
there is an infinite directed path through this population with all edges 
male.  By construction, male edges never step in a forbidden direction.
\end{proof}

\section{Further Questions}

If $s$ is an avoidable gender sequence, to what extent can we find populations avoiding $s$
which are \emph{universal} among all such, in a way analogous to that described by Cherlin and Shelah \cite{cherlin}?

Given a possibly avoidable sequence, can the infinite gendered populations which realize that sequence
be characterized in some way?  Of particular interest would be a characterization in terms of ordinal
numbers, similar to R.~Schmidt's characterization of rayless graphs
(see Halin \cite{halin}).

But perhaps the most important question remaining is,
what are the biologically unavoidable sequences?  Are there any which are not eventually periodic?

\end{document}